\documentclass[11pt]{amsart}
\title[Cohomologies of certain solvable groups]
{Central theorems for cohomologies of certain solvable groups}

\author{Hisashi Kasuya}

\usepackage{amssymb}
\usepackage{amsmath}
\usepackage{amscd}
\usepackage{amstext}
\usepackage{amsfonts}
\usepackage[all]{xy}

\theoremstyle{plain}

\theoremstyle{plain}

\theoremstyle{plain}

\theoremstyle{plain}
\newtheorem{theorem}{Theorem}[section] 
\theoremstyle{remark}
\newtheorem{remark}{Remark}
\theoremstyle{remark}
\newtheorem{Important note}[theorem]{Important note}
\theoremstyle{Main result}
\newtheorem{main result}{Main result}
\theoremstyle{lemma}
\newtheorem{lemma}[theorem]{Lemma}
\theoremstyle{definition}
\newtheorem{definition}[theorem]{Definition}
\theoremstyle{proposition}
\newtheorem{proposition}[theorem]{Proposition}
\theoremstyle{corollary}
\newtheorem{corollary}[theorem]{Corollary}
\theoremstyle{remark}

\address[Hisashi Kasuya]{Department of Mathematics, Tokyo Institute of Technology, 1-12-1, O-okayama, Meguro, Tokyo 152-8551, Japan}
\email{kasuya@math.titech.ac.jp}


\newcommand{\C}{\mathbb{C}}
\newcommand{\R}{\mathbb{R}}

\newcommand{\Q}{\mathbb{Q}}
\newcommand{\Z}{\mathbb{Z}}
\newcommand{\g}{\frak{g}}
\newcommand{\n}{\frak{n}}

\begin{document} 

\maketitle
\begin{abstract}
We show that the group cohomology of torsion-free virtually polycyclic groups and the continuous cohomology of simply connected solvable Lie groups can be computed by the rational cohomology of algebraic groups.
Our results are generalizations of certain results on the cohomology of solvmanifolds and infra-solvmanifolds.
Moreover as an application of our results, we give a new proof of the  surprising   cohomology vanishing theorem given by Dekimpe-Igodt.
\end{abstract}

\section{Main theorems}
For a set $M$ and a group $G$, we suppose that  $G$ acts on $M$.
Then  we denote by $M^{G}$ the set of the $G$-invariant elements.
Let $k$ be a sub-field of $\C$. 
A $k$-defined algebraic group $\mathcal G$ is a Zariski-closed subgroup of $GL_{n}(\C)$ which is defined by polynomials with coefficients in $k$.
For a $k$-defined algebraic group $\mathcal G$,  we denote by ${\mathcal  U}({\mathcal G})$ the unipotent  radical of  $\mathcal G$ and we denote by ${\mathcal G}(k)$ the group of $k$-points.
In this paper,   a  $k$-vector space $V$ is a vector space over $\C$ with a $k$-structure $V_{k}$ as in \cite{Bor}.
Let $V$ be a  $k$-vector space.
We consider the following cohomologies.
\begin{enumerate}

\item For a group $\Gamma$, assuming that $V_{k}$ is a $\Gamma$-module, we define the group cohomology $H^{\ast}(\Gamma,V_{k})={\rm Ext}^{\ast}_{\Gamma}(k,V_{k})$ as in \cite{Bro}.

\item For a connected Lie  group $G$, assuming  that $k= \R$ or $\C$  and $V_{k}$ is a topological $G$-module ,
we define the continuous cohomology $H^{\ast}_{c}(G,V_{k})={\rm Ext}^{\ast}_{G}(k,V_{k})$ as in \cite[Section IX]{BW}.

\item For a $k$-defined algebraic  group $\mathcal G$, assuming that $V$ is a $k$-rational  $\mathcal G$-module, we define the rational cohomology $H^{\ast}_{r}(\mathcal G,V_{k})={\rm Ext}^{\ast}_{\mathcal G(k)}(k,V_{k})$ as in  \cite{Hoc} and \cite{Jan}.

\end{enumerate}
We suppose that we have an inclusion $\Gamma\subset {\mathcal G}(k)$ (resp. $G\subset {\mathcal G}(k)$) as a subgroup (resp. Lie-subgroup).
Let $V$ be a $k$-rational $\mathcal G$-module.
 Then the inclusion induces  the homomorphism $H^{\ast}_{r}({\mathcal G},V_{k})\to H^{\ast}(\Gamma,V_{k})$ (resp. $H^{\ast}_{r}({\mathcal G},V_{k})\to H^{\ast}_{c}(G,V_{k})$).

A group $\Gamma$ is polycyclic if it admits a sequence 
\[\Gamma=\Gamma_{0}\supset \Gamma_{1}\supset \cdot \cdot \cdot \supset \Gamma_{k}=\{ e \}\]
of subgroups such that each $\Gamma_{i}$ is normal in $\Gamma_{i-1}$ and $\Gamma_{i-1}/\Gamma_{i}$ is cyclic.
We denote ${\rm rank}\,\Gamma=\sum_{i=1}^{i=k}{\rm rank}\, \Gamma_{i-1}/\Gamma_{i}$.

In this paper we study the group cohomology $H^{\ast}(\Gamma,V_{k})$ of a torsion-free virtually polycyclic group $\Gamma$, the continuous cohomology $H^{\ast}_{c}(G,V)$ of a simply connected solvable Lie group $G$ and the rational cohomology $H_{r}^{\ast}(\mathcal G,V_{k})$ of an algebraic group $\mathcal G$ which contains $\Gamma$ or $G$ as a Zariski-dense subgroup.
We have:
\begin{lemma}{\rm(\cite[Lemma 4.36.]{R})}\label{hut}
Let $\Gamma$ (resp. $G$) be a torsion-free virtually polycyclic group (simply connected solvable Lie group).
 For a finite-dimensional representation $\rho:\Gamma\to GL_{n}(\C)$ (resp. $\rho:G\to GL_{n}(\C)$), denoting by $\mathcal G$ the Zariski-closure of $\rho(\Gamma)$ (resp. $\rho(G)$)  in $ GL_{n}(\C)$, we have $\dim \mathcal U(\mathcal G)\le {\rm rank}\, \Gamma$ (resp. $\dim \mathcal U(\mathcal G)\le\dim G$),
where we denote by ${\mathcal  U}({\mathcal G})$ the unipotent radical of  $\mathcal G$.
\end{lemma}

By this lemma, we consider the following definition.
\begin{definition}
Let $\Gamma$ (resp. $G$) be a torsion-free virtually polycyclic group  (resp. simply connected solvable Lie group). 
 For an algebraic group $\mathcal G$, a  representation $\rho:\Gamma\to \mathcal G$ (resp. $\rho:G\to \mathcal G$)   is called {\em full} if the image $\rho(\Gamma)$ (resp. $\rho(G)$) is Zariski-dense in $ \mathcal G$ and we have $\dim \mathcal U(\mathcal G)= {\rm rank}\, \Gamma$ (resp. $\dim \mathcal U(\mathcal G)=\dim G$).
\end{definition}

In this paper, we prove the following theorem.
\begin{theorem}\label{poly1}
Let $\Gamma$ be a torsion-free virtually polycyclic group.
We suppose that for a subfield $k\subset \C$ we have a $k$-defined algebraic group $\mathcal G$ and an inclusion $\Gamma\subset {\mathcal G}(k)$ which is a full representation.
Let $V$ be a $k$-rational $\mathcal G$-module.
Then  the inclusion $\Gamma\subset {\mathcal G}(k)$ induces a cohomology isomorphism
\[H^{\ast}_{r}({\mathcal G},V_{k})\cong H^{\ast}(\Gamma,V_{k}).
\]
\end{theorem}
We also  prove the following theorem.
\begin{theorem}\label{lie1}
Let $G$ be a  simply connected solvable Lie group.
We suppose that for a field $k= \R$ or $\C$ we have a $k$-defined algebraic group $\mathcal G$ and an inclusion $G\subset {\mathcal G}(k)$ which is a full representation.
Let $V$ be a  rational $\mathcal G$-module.
Then  the inclusion $G\subset {\mathcal G}(k)$ induces a cohomology isomorphism
\[H^{\ast}_{r}({\mathcal G},V_{k})\cong H^{\ast}_{c}(G,V_{k}).
\]
\end{theorem}

By Theorem \ref{poly1}, for any torsion-free polycyclic group $\Gamma$ and any finite-dimensional representation $\rho:\Gamma\to GL(V_{k}) $, we can describe the group  cohomology $H^{\ast}(\Gamma,V_{k})$ as a subspace of the cohomology of certain nilpotent Lie algebra determined by $\Gamma$ (Corollary \ref{pppis}, see also Theorem \ref{BEXT}).
The above main theorems are generalizations of the classical  results on the cohomology of nilmanifolds and solvmanifolds given by Nomizu and Mostow and Baues' recent result on the cohomology of compact aspherical manifolds with torsion-free virtually polycyclic fundamental groups. 
We see that the theorems imply such results as corollaries
(Section \ref{cor}).

Moreover, as an application of our theorem, we will give a new proof of the  surprising   cohomology vanishing theorem given by Dekimpe-Igodt in \cite{Deki}
(Section \ref{DDEKimp}).
For constructing bounded polynomial  crystallographic actions of  polycyclic groups Dekimpe and Igodt prove the vanishing of certain cohomology.
 By our proof, we can regard such vanishing as a conclusion of a fundamental lemma on  cohomology of algebraic groups.

 {\bf  Acknowledgements.} 

The author  would  like to thank   Karel Dekimpe for  helpful comments.
This research is supported by JSPS Grant-in-Aid for Research Activity start-up.

\section{Proof of Main Theorems}\label{prf}

\subsection{First step}

\begin{lemma}\label{ra1}
Let $\Gamma=\Z$.
We suppose that for a subfield $k\subset \C$ we have a $k$-defined algebraic group $\mathcal G$ and an inclusion $\Gamma\subset {\mathcal G}(k)$ which is a full representation.
Let $V$ be a $k$-rational $\mathcal G$-module.
Then  the inclusion $\Gamma\subset {\mathcal G}(k)$ induces a cohomology isomorphism
\[H^{\ast}_{r}({\mathcal G},V_{k})\cong H^{\ast}(\Gamma,V_{k}).
\]
\end{lemma}
\begin{proof}
For a rational $\mathcal G$-module, we have $V=\bigcup_{i} V^{i}$ for finite-dimensional rational $k$-submodules $V^{i}$.
Then by \cite{Brf} we have \[H^{\ast}(\Gamma,\bigcup_{i} {V^{i}}_{k})\cong \varinjlim H^{\ast}(\Gamma,  {V^{i}}_{k})
\]
and by \cite[Part 1. Lemma 4.17]{Jan} we have 
\[H^{\ast}_{r}(\mathcal G, \bigcup_{i} {V^{i}}_{k})\cong \varinjlim H^{\ast}_{r}(\mathcal G, {V^i}_{k}).
\]
Hence we can assume that the rational $\mathcal G$-module is finite-dimensional.

For $\ast\not=0,1$ we have $H^{\ast}(\Gamma,V_{k})=0$ and  $H^{\ast}_{r}(\mathcal G, V_{k})=0$.
Since $\Gamma$ is a Zariski-dense subgroup in $\mathcal G$,
we have $H^{0}(\Gamma,V_{k})=V^{\Gamma}=V_{k}^{\mathcal G(k)}=H^{0}(\mathcal G,V_{k})$.
Hence it is sufficient to show that the map $H^{1}_{r}(\mathcal G,V_{k})\to H^{1}(\Gamma, V_{k})$ is an isomorphism.

Since $\Gamma$ is Abelian, $\mathcal G$ is also Abelian and hence we have $\mathcal G=\mathcal G_{s}\times {\mathcal U}(\mathcal G)$ where $\mathcal G_{s}=\{ g\in {\mathcal G}\vert\, g\,\, {\rm  is}\,\,{\rm semi-simple}\} $ (see \cite[Theorem 4.7]{Bor}).
For the representation $\phi:\mathcal G\to GL(V)$, the decomposition $\mathcal G=\mathcal G_{s}\times {\mathcal U}(\mathcal G)$ gives
the  decomposition $\phi=\phi_{s}\times \phi_{u}$.  
Let $V^{0}=\{v\in V\vert \forall g\in \mathcal G\,\,\, \phi_{s}(g)v=v \}$. 
Then $V^{0}$ is a $k$-rational sub-module of $V$ for  both the representations $\phi_{s}$ and $\phi$.
 We can take  a complement $V^{1}$ such that we have  a direct sum $V=V^{0}\oplus V^{1}$ of sub-module for  both $\phi_{s}$ and $\phi$.

We first show $H^{1}_{r}(\mathcal G,  V^{1}_{k})=0$ and $H^{1}(\Gamma, V^{1}_{k})=0$.
It is sufficient to show the case $k=\C$.
Then we have the sequence $V^{1}=V^{1,0}\supset V^{1,1}\supset \dots\supset V^{1, m}$ of sub-modules such that $\dim V_{1,i}/V_{1, i+1}=1$.
By induction, It is sufficient to show the case $V^{1}$ is a non-trivial $1$-dimensional module.
In this case it is easy to show 
 $H^{1}(\Gamma, V^{1}_{k})=0$.
Since $V^{1}$ is $1$-dimensional, the unipotent radical ${\mathcal U}(\mathcal G)$ acts trivially on $V^{1}$.
Let ${\frak u}_{k}$ be the $k$-Lie algebra of ${\mathcal U}(\mathcal G)$.
Then we have 
\[H^{1}_{r}(\mathcal G,  V^{1}_{k})\cong H^{1}({\frak u}_{k},  V^{1}_{k})^{\mathcal G_{s}(k)},
\]
see Theorem \ref{hocun}.
Since $\mathcal G_{s}$ acts non-trivially on $V^{1}$,  ${\mathcal U}(\mathcal G)$ acts trivially on $V^{1}$ and the decomposition $\mathcal G=\mathcal G_{s}\times {\mathcal U}(\mathcal G)$ is a direct product,
we have 
\[H^{1}({\frak u}_{k},  V^{1}_{k})^{\mathcal G_{s}(k)}={\frak u}_{k}^{\ast}\otimes (V^{1}_{k})^{\mathcal G_{s}(k)}=0.\]

We show that  the inclusion $\Gamma\subset {\mathcal G}(k)$ induces a cohomology isomorphism
\[H^{1}_{r}({\mathcal G},V^{0}_{k})\cong H^{1}(\Gamma,V^{0}_{k}).
\]
Since $\phi$ is a unipotent representation on $V^{0}$, we have 
$V^{0}=V^{0,0}\supset V^{0,1}\supset \dots\supset V^{0,m}$ of sub-modules such that $ V^{0,i}/V^{0, i+1}$ is the $1$-dimensional trivial module.
By induction, it is sufficient to show the case $V^{0}_{k}=k$ as the trivial module.
By the projection $\mathcal G=\mathcal G_{s}\times {\mathcal U}(\mathcal G) \to {\mathcal U}(\mathcal G)$, we have the Zariski-dense inclusion $\Gamma\subset {\mathcal U}(\mathcal G)(k)$.
By $\Gamma=\Z$, we have ${\mathcal U}(\mathcal G)(k)=k$ as the additive group and $\Gamma$ is embedded as $a\Z$ for some $a\in k$.
We can easily show that the inclusion $\Gamma\subset {\mathcal U}(\mathcal G)$ induces  a cohomology isomorphism
\[H_{r}^{1}(\mathcal U(\mathcal G),k)\cong H^{1}(\Gamma,k)
\]
see \cite[Part I. 4.21 Lemma]{Jan}.
By the direct product $\mathcal G=\mathcal G_{s}\times {\mathcal U}(\mathcal G)$, we have 
\[H_{r}^{1}(\mathcal G,k)\cong H_{r}^{1}(\mathcal U(\mathcal G),k)^{\mathcal G_{s}(k)}=H_{r}^{1}(\mathcal U(\mathcal G),k),\]
see Theorem \ref{unirediso}.
By this isomorphism, 
we can prove that 
the inclusion $\Gamma\subset {\mathcal G}(k)$ induces a cohomology isomorphism
\[H^{1}_{r}({\mathcal G},V^{0}_{k})\cong H^{1}(\Gamma,V^{0}_{k})
\]
and the lemma follows.
\end{proof}

By a similar proof, we have the following lemma.
(See \cite[Chapter IX]{BW} for the commutativity of inductive limits for the continuous cohomology of Lie groups.)
\begin{lemma}\label{d1r}
Let $G=\R$.
We suppose that for a field $k= \R$ or $\C$ we have a $k$-defined algebraic group $\mathcal G$ and an inclusion $G\subset {\mathcal G}(k)$ which is a full representation.
Let $V$ be a  $k$-rational $\mathcal G$-module.
Then  the inclusion $G\subset {\mathcal G}(k)$ induces a cohomology isomorphism
\[H^{\ast}_{r}({\mathcal G},V_{k})\cong H^{\ast}_{c}(G,V_{k}).
\]
\end{lemma}

\subsection{Proof of Main Theorems}

\begin{proof}[Proof of Theorem \ref{poly1}]
We proceed by induction on ${\rm rank}\, \Gamma$.
By Lemma \ref{ra1}, in the case ${\rm rank}\, \Gamma=1$ the statement follows.

$\Gamma$ admits a  finite index normal polycyclic subgroup $\Gamma^{\prime}$ such that $\Gamma^{\prime}$ admits a normal subgroup $\Gamma^{\prime\prime}$ such that $\Gamma^{\prime}/\Gamma^{\prime\prime}$ is  infinite cyclic (see \cite{R}).
Denote by $\mathcal G^{\prime}$ and $\mathcal H$ the Zariski-closures of $\Gamma^{\prime}$ and $\Gamma^{\prime\prime}$ respectively in $\mathcal G$.
We notice that $\Gamma^{\prime}/\Gamma^{\prime\prime}\to \mathcal G^{\prime}/\mathcal H$ has a Zariski-dense image.

Since $\Gamma^{\prime}$ is a finite index subgroup of $\Gamma$, we have ${\rm rank}\, \Gamma={\rm rank}\,\Gamma^{\prime}$.
By the fullness, we have $\dim \mathcal U(\mathcal G)={\rm rank}\, \Gamma$.
Since $\Gamma/\Gamma^{\prime}$ is finite, $\mathcal G/\mathcal G^{\prime}$ is also finite and so we have
$\mathcal U(\mathcal G)=\mathcal U(\mathcal G^{\prime})$.
Hence we have $\dim \mathcal U(\mathcal G^{\prime})={\rm rank}\,\Gamma^{\prime}$ and so $\Gamma^{\prime}\to \mathcal G^{\prime}$ is also full.

We have 
\[{\rm rank}\,\Gamma^{\prime}={\rm rank}\,\Gamma^{\prime\prime}+{\rm rank}\,\Gamma^{\prime}/\Gamma^{\prime\prime}\] and 
\[\dim  \mathcal U(\mathcal G^{\prime})=\dim \mathcal U(\mathcal H)+ \dim \mathcal U(\mathcal G^{\prime}/\mathcal H).\]
By Lemma \ref{hut}, we have $\dim \mathcal U(\mathcal H)\le {\rm rank}\,\Gamma^{\prime\prime}$ and $\dim \mathcal U(\mathcal G^{\prime}/\mathcal H)\le {\rm rank}\,\Gamma^{\prime}/\Gamma^{\prime\prime}$.
By $\dim \mathcal U(\mathcal G^{\prime})={\rm rank}\,\Gamma^{\prime}$, these relations imply
$\dim \mathcal U(\mathcal H)= {\rm rank}\,\Gamma^{\prime\prime}$ and $\dim \mathcal U(\mathcal G^{\prime}/\mathcal H)= {\rm rank}\,\Gamma^{\prime}/\Gamma^{\prime\prime}$.
Hence  both representations $\Gamma^{\prime\prime}\to \mathcal H$ and $\Gamma^{\prime}/\Gamma^{\prime\prime} \to \mathcal G^{\prime}/\mathcal H$ are full.

Take $\Delta=\Gamma^{\prime}\cap \mathcal H$.
By $\Gamma^{\prime\prime}\subset\Delta$, we have
${\rm rank}\,\Gamma^{\prime}/\Delta \le {\rm rank}\,\Gamma^{\prime}/\Gamma^{\prime\prime}$.
Since $\Gamma^{\prime }/\Delta$ is Zariski-dense in $\mathcal G^{\prime}/\mathcal H$, by Lemma \ref{hut}, we have $ \dim \mathcal U(\mathcal G^{\prime}/\mathcal H)\le {\rm rank}  \Gamma^{\prime }/\Delta$.
By $\dim \mathcal U(\mathcal G^{\prime}/\mathcal H)={\rm rank}\,\Gamma^{\prime}/\Gamma^{\prime\prime}$,
we obtain ${\rm rank}\,\Gamma^{\prime}/\Delta ={\rm rank}\,\Gamma^{\prime}/\Gamma^{\prime\prime}$.
Since we have $\Gamma^{\prime}/\Delta\cong (\Gamma^{\prime}/\Gamma^{\prime\prime})/(\Delta/\Gamma^{\prime\prime})$
and $\Gamma^{\prime}/\Gamma^{\prime\prime}$ is infinite cyclic, we have $\Delta/\Gamma^{\prime\prime}=1$ and so $\Delta=\Gamma^{\prime\prime}$.

Now  we have the commutative diagram
\[\xymatrix{
1\ar[r]& \Delta\ar[r]\ar[d]&\Gamma^{\prime}\ar[r]\ar[d]&\Gamma^{\prime}/\Delta\ar[r]\ar[d]&1,\\
1\ar[r]& \mathcal H\ar[r]&\mathcal G^{\prime}\ar[r]&\mathcal G^{\prime}/\mathcal H \ar[r]&1
 }\]
such that $\Delta$,  $\Gamma^{\prime}$ and $\Gamma^{\prime}/\Delta$ are Zariski-dense subgroups in $\mathcal H$, $\mathcal G^{\prime}$ and $\mathcal G^{\prime}/\mathcal H$ respectively.
The groups $\Delta$,  $\Gamma^{\prime}$ and $\Gamma^{\prime}/\Delta$ are torsion-free polycyclic and the inclusions $\Delta\subset \mathcal H$,  $\Gamma^{\prime}\subset \mathcal G^{\prime}$ and $\Gamma^{\prime}/\Delta\subset \mathcal G^{\prime}/\mathcal H$ are full representations.

Consider the spectral sequence
$E_{\ast}^{\ast,\ast}(\Gamma^{\prime},\Delta,V_{k})$
of the group extension
\[\xymatrix{
1\ar[r]& \Delta\ar[r]&\Gamma^{\prime}\ar[r]&\Gamma^{\prime}/\Delta\ar[r]&1,}\]
as in \cite{HS} and the spectral sequence $\, _{r}E^{\ast,\ast}_{\ast}(\mathcal G^{\prime},\mathcal H,V_{k})$ of the $k$-defined  algebraic group extension
\[\xymatrix{
1\ar[r]& \mathcal H\ar[r]&\mathcal G^{\prime}\ar[r]&\mathcal G^{\prime}/\mathcal H \ar[r]&1
 }\]
as in \cite[Proposition 6.6]{Jan}.
Then  the commutative diagram
\[\xymatrix{
1\ar[r]&\Delta\ar[r]\ar[d]&\Gamma^{\prime}\ar[r]\ar[d]&\Gamma^{\prime}/\Delta\ar[r]\ar[d]&1,\\
1\ar[r]& \mathcal H\ar[r]&\mathcal G^{\prime}\ar[r]&\mathcal G^{\prime}/\mathcal H \ar[r]&1
 }\]
 induces
 a homomorphism $\, _{r}E^{\ast,\ast}_{\ast}(\mathcal G^{\prime},\mathcal H,V_{k})\to E_{\ast}^{\ast,\ast}(\Gamma^{\prime},\Delta,V_{k})$
such that we have the commutative diagram
\[\xymatrix{
\, _{r}E^{\ast,\ast}_{2}(\mathcal G^{\prime},\mathcal H,V)\ar[r]\ar^{\cong}[d]& E_{2}^{\ast,\ast}(\Gamma^{\prime},\Delta,V_{k})\ar^{\cong}[d] \\
H^{\ast}_{r}\left(\mathcal G^{\prime}/\mathcal H, H^{\ast}_{r}(\mathcal H,V_{k})\right)\ar[r] &H^{\ast}\left(\Gamma^{\prime}/\Delta, H^{\ast}(\Delta,V_{k})\right).
 }\]

By induction hypothesis,  the inclusion $\Delta\subset \mathcal H$ induces a cohomology isomorphism $H^{\ast}_{r}(\mathcal H,V_{k})\cong H^{\ast}(\Delta,V_{k})$  and the inclusion $\Gamma^{\prime}/\Delta\subset \mathcal G^{\prime}/\mathcal H$  induces a cohomology isomorphism $H^{\ast}_{r}\left(\mathcal G/\mathcal H, H^{\ast}_{r}(\mathcal H,V_{k})\right)\cong H^{\ast}\left(\Gamma/\Delta, H^{\ast}_{r}(\mathcal H,V_{k})\right)$.
Hence the homomorphism $\, _{r}E^{\ast,\ast}_{2}(\mathcal G^{\prime},\mathcal H,V_{k})\to E_{2}^{\ast,\ast}(\Gamma^{\prime},\Delta,V_{k})$ is an isomorphism.
 By \cite[Theorem 3.5]{Mc}, we can show the isomorphism 
\[H^{\ast}_{r}(\mathcal G^{\prime},V_{k})\cong H^{\ast}( \Gamma^{\prime},V_{k}).
\]
Since $\Gamma/\Gamma^{\prime}$ is a finite group,  the map $\Gamma/\Gamma^{\prime}\to  \mathcal G/\mathcal G^{\prime}$ is surjective by the Zariski-density.
Hence $\mathcal G/\mathcal G^{\prime}$ is also finite.
We have isomorphisms 
$H^{\ast}_{r}(\mathcal G,V_{k})\cong H^{\ast}_{r}(\mathcal G^{\prime},V_{k})^{\mathcal G/\mathcal G^{\prime}(k)}$
and $H^{\ast}( \Gamma^{\prime},V_{k})^{\Gamma/\Gamma^{\prime}}\cong H^{\ast}( \Gamma,V_{k})$
 and the induced map \[H^{\ast}_{r}(\mathcal G,V_{k})\to  H^{\ast}( \Gamma,V_{k})\] is identified with the map
\[ H^{\ast}_{r}(\mathcal G^{\prime},V_{k})^{\mathcal G/\mathcal G^{\prime}(k)}\to  H^{\ast}( \Gamma^{\prime},V_{k})^{\Gamma/\Gamma^{\prime}}
\]
associated with the maps $\Gamma^{\prime}\to \mathcal G^{\prime}$ and $\Gamma/\Gamma^{\prime}\to  \mathcal G/\mathcal G^{\prime}$.
Since the map  $\Gamma/\Gamma^{\prime}\to  \mathcal G/\mathcal G^{\prime}$ is surjective and we have the isomorphism 
\[H^{\ast}_{r}(\mathcal G^{\prime},V_{k})\cong H^{\ast}( \Gamma^{\prime},V_{k})
\]
as above, we can show that  the map 
\[ H^{\ast}_{r}(\mathcal G^{\prime},V_{k})^{\mathcal G/\mathcal G^{\prime}(k)}\to  H^{\ast}( \Gamma^{\prime},V_{k})^{\Gamma/\Gamma^{\prime}}
\]
is an isomorphism.
The injectivity  is obvious.
By the surjectivity of   $\Gamma/\Gamma^{\prime}\to  \mathcal G/\mathcal G^{\prime}$,
we can easily check that if $\iota (a)\in H^{\ast}( \Gamma^{\prime},V_{k})^{\Gamma/\Gamma^{\prime}}
$ for $a\in H^{\ast}_{r}(\mathcal G^{\prime},V_{k})$, then we have $a\in H^{\ast}_{r}(\mathcal G^{\prime},V_{k})^{\mathcal G/\mathcal G^{\prime}(k)}$ 
where we denote by $\iota$ the isomorphism $H^{\ast}_{r}(\mathcal G^{\prime},V_{k})\cong H^{\ast}( \Gamma^{\prime},V_{k})$.
This implies the surjectivity of  $ H^{\ast}_{r}(\mathcal G^{\prime},V_{k})^{\mathcal G/\mathcal G^{\prime}(k)}\to  H^{\ast}( \Gamma^{\prime},V_{k})^{\Gamma/\Gamma^{\prime}}$.
Hence the theorem  follows.
\end{proof}

\begin{proof}[Proof of Theorem \ref{lie1}]
We proceed by induction on $\dim G$.
By Lemma \ref{d1r}, in the case ${\rm dim}\, G=1$ the statement follows. 

We have a normal subgroup $G^{\prime\prime}\subset G$ such that $G/ G^{\prime\prime}=\R$.
Denote by $\mathcal H$ the Zariski-closure of  $G^{\prime\prime}$  in $\mathcal G$.
We have 
\[\dim G=\dim G^{\prime\prime}+\dim G/G^{\prime\prime}
\]
and 
\[\dim \mathcal U(\mathcal G)=\dim \mathcal U(\mathcal H)+\dim \mathcal U(\mathcal G/\mathcal H).
\]
By Lemma \ref{hut}, we have $\dim \mathcal U(\mathcal H)\le \dim G^{\prime\prime}$ and $\dim \mathcal U(\mathcal G/\mathcal H)\le \dim G/G^{\prime\prime}$.
By the fullness, we have $\dim G=\dim \mathcal U(\mathcal G)$.
These relations imply  $\dim \mathcal U(\mathcal H)= \dim G^{\prime\prime}$ and $\dim \mathcal U(\mathcal G/\mathcal H)= \dim G/G^{\prime\prime}$.
Hence both representations $G^{\prime\prime}\to \mathcal H$ and $G/G^{\prime\prime} \to \mathcal G/\mathcal H$ are full.

Take $H=G\cap \mathcal H$.
Then $H$ is a closed normal subgroup in $G$ and hence $H$ and $G/H$ are  simply connected (see \cite{OV2}).
By $G^{\prime\prime}\subset H$, we have $\dim G/H\le \dim G/G^{\prime\prime}$.
Since  $G/H$ is Zariski-dense in $\mathcal G/\mathcal H$,
by  Lemma \ref{hut}, we have $\dim \mathcal U(\mathcal G/\mathcal H)\le \dim G/H$.
Hence, by $\dim \mathcal U(\mathcal G/\mathcal H)= \dim G/G^{\prime\prime}$, we obtain $\dim G/H=\dim G/G^{\prime\prime}$.
Thus $\dim H=\dim G^{\prime\prime}$ and so $H=G^{\prime\prime}$.

 We have the commutative diagram
\[\xymatrix{
1\ar[r]& H\ar[r]\ar[d]&G\ar[r]\ar[d]&G/H\ar[r]\ar[d]&1,\\
1\ar[r]& \mathcal H\ar[r]&\mathcal G\ar[r]&\mathcal G/\mathcal H \ar[r]&1
 }\]
such that $H$, $G$ and $G/H$ are Zariski-dense subgroups in $\mathcal H$, $\mathcal G$ and $\mathcal G/\mathcal H$ respectively.
The groups  $H$,  $G$ and $G/H$ are simply connected solvable Lie groups and the inclusions $H\subset \mathcal H$,  $G\subset \mathcal G$ and $G/H\subset \mathcal G/\mathcal H$ are full representations.

Consider the spectral sequence 
$\,_{c}E_{\ast}^{\ast,\ast}(G,H,V_{k})$
of the Lie  group extension
\[\xymatrix{
1\ar[r]& H\ar[r]&G\ar[r]&G/H\ar[r]&1,}\]
as in \cite[Chapter IX. Theorem 4.3]{BW} or \cite[Proof of Theorem 7.1]{HM} and the spectral sequence $\, _{r}E^{\ast,\ast}_{\ast}(\mathcal G,\mathcal H,V_{k})$ of the $k$-defined algebraic group extension
\[\xymatrix{
1\ar[r]& \mathcal H\ar[r]&\mathcal G\ar[r]&\mathcal G/\mathcal H \ar[r]&1
 }\]
as in \cite[Proposition 6.6]{Jan}.
Then the commutative diagram
\[\xymatrix{
1\ar[r]&H\ar[r]\ar[d]&G\ar[r]\ar[d]&G/H\ar[r]\ar[d]&1,\\
1\ar[r]& \mathcal H\ar[r]&\mathcal G\ar[r]&\mathcal G/\mathcal H \ar[r]&1
 }\]
 induces
 a homomorphism $\, _{r}E^{\ast,\ast}_{\ast}(\mathcal G,\mathcal H,V_{k})\to \,_{c}E_{\ast}^{\ast,\ast}(G,H,V_{k})$
such that we have the commutative diagram
\[\xymatrix{
\, _{r}E^{\ast,\ast}_{2}(\mathcal G,\mathcal H,V_{k})\ar[r]\ar^{\cong}[d]& \,_{c}E_{2}^{\ast,\ast}(G,H,V_{k})\ar^{\cong}[d] \\
H^{\ast}_{r}\left(\mathcal G/\mathcal H, H^{\ast}_{r}(\mathcal H,V_{k})\right)\ar[r] &H^{\ast}_{c}\left(G/H, H_{c}^{\ast}(H,V_{k})\right).
 }\]

By induction hypothesis,  the inclusion $H\subset \mathcal H$ induces a cohomology isomorphism $H^{\ast}_{r}(\mathcal H,V)\cong H^{\ast}_{c}(H,V)$  and the inclusion $G/H\subset \mathcal G/\mathcal H$  induces a cohomology isomorphism $H^{\ast}_{r}\left(\mathcal G/\mathcal H, H^{\ast}_{r}(\mathcal H,V_{k})\right)\cong H^{\ast}_{c}\left(G/H, H^{\ast}_{r}(\mathcal H,V_{k})\right)$.
Hence the homomorphism $\, _{r}E^{\ast,\ast}_{2}(\mathcal G,\mathcal H,V_{k})\to \,_{c}E_{2}^{\ast,\ast}(G,H,V_{k})$ is an isomorphism.
 By \cite[Theorem 3.5]{Mc}, we have 
\[H^{\ast}_{r}(\mathcal G,V_{k})\cong H^{\ast}_{c}( G,V_{k}).
\]
\end{proof}

\section{Full representations and algebraic hulls}\label{sefull}

An algebraic group $\mathcal G$ is {\it minimal} if the centralizer $Z_{\mathcal G}({\mathcal U}(\mathcal G))$ of ${\mathcal  U}(\mathcal G)$ is contained in  ${\mathcal  U}({\mathcal G})$ (see \cite[Section 1]{AN} for  the meaning of minimal).

Let $\Gamma$ (resp. $G$) be a torsion-free virtually polycyclic group (resp. simply connected solvable Lie group).
 \begin{theorem} \rm (\cite[Theorem A.1, Corollary A.3]{B}
\cite[Proposition 4.40, Lemma 4.41]{R})
There exists a $\Q$-defined (resp. $\R$-defined) minimal algebraic group $\mathcal G$ with an inclusion $\Gamma\subset {\mathcal G}(\Q)$ (resp. $G\subset {\mathcal G}(\R)$) which is a full representation.

Moreover such  $\mathcal G$ is unique up to isomorphism of  $\Q$-defined (resp. $\R$-defined)  algebraic groups.
\end{theorem}

We call the  $\Q$-defined (resp. $\R$-defined) algebraic group  as in this theorem  the {\it algebraic hull} of $\Gamma$ (resp. $G$).
We denote it by ${\mathcal A}_{\Gamma} $ (resp. ${\mathcal A}_G$).

Let $k$ be a  subfield $k\subset \C$ (resp. $k= \R$ or $\C$).  Let $\rho :\Gamma\to GL(V_{k})$ (resp. $\rho :G\to GL(V_{k})$) be a finite-dimensional representation.
\begin{proposition}{\rm(\cite[Proposition A.6]{B})}\label{miniex}
$\rho$ can be extended to a $k$-rational representation of  ${\mathcal A}_{\Gamma}$ (resp. ${\mathcal A}_{G}$) if the Zariski-closure of  the image of $\rho$ is  minimal. 
\end{proposition}

In general,  a finite-dimensional representation $\rho$ can not be extended to a $k$-rational representation of   ${\mathcal A}_{\Gamma}$.
In this paper we consider the following construction.
\begin{definition}
Consider the inclusion $i_{\Gamma}:\Gamma\hookrightarrow \mathcal A_{\Gamma}(\Q)$ (resp. $i_{G}:G\hookrightarrow \mathcal A_{G}(\R)$)
and the representation $i_{\Gamma}\oplus \rho :\Gamma\to \mathcal A_{\Gamma}(k)\times GL(V_{k})$ (resp. $i_{G}\oplus \rho :G\to \mathcal A_{G}(k)\times GL(V_{k})$).
Then we denote by $\mathcal G_{\Gamma}^{\rho}$ (resp. $\mathcal G_{G}^{\rho}$) the Zariski-closure of the image of $i_{\Gamma}\oplus \rho$ (resp. $i_{G}\oplus \rho$)
and we call it the $\rho$-{\it relative algebraic hull} of $\Gamma$ (resp. $G$).
\end{definition}

We can easily show the following properties.
\begin{proposition}\label{relproa}
\begin{itemize}

\item $\rho :\Gamma\to GL(V_{k})$  (resp. $\rho :G\to GL(V_{k})$) can be extended to a $k$-rational representation $\mathcal G_{\Gamma}^{\rho}\to GL(V)$ (resp. $\mathcal G_{G}^{\rho}\to GL(V)$).

\item We have an isomorphism $\mathcal U(\mathcal G_{\Gamma}^{\rho})\cong \mathcal U(\mathcal A_{\Gamma}^{\rho})$ of $k$-defined algebraic groups.

\item $i_{\Gamma}\oplus \rho:\Gamma\to \mathcal G_{\Gamma}^{\rho}$ (resp. $i_{G}\oplus \rho: G\to \mathcal G_{G}^{\rho}$)  is an injective full representation.

\item If the Zariski-closure of $\rho(\Gamma)$ (resp. $\rho(G)$ ) is minimal,
then $\mathcal G_{\Gamma}^{\rho}= \mathcal A_{\Gamma}$.
(resp. $\mathcal G_{G}^{\rho}= \mathcal A_{G}$.)
\end{itemize}

\end{proposition}
\begin{proof}

The restriction $\mathcal G_{\Gamma}^{\rho}\to GL(V)$ of the projection $\mathcal A_{\Gamma}\times GL(V)\to  GL(V)$ gives the first assertion.

Consider the restriction $\mathcal G_{\Gamma}^{\rho}\to \mathcal A_{\Gamma}$ of the projection $\mathcal A_{\Gamma}\times GL(V)\to \mathcal A_{\Gamma}$.
Since the inclusion $i_{\Gamma}:\Gamma\to \mathcal A_{\Gamma}$ has the Zariski-dense image,
the map $\mathcal G_{\Gamma}^{\rho}\to \mathcal A_{\Gamma}$ is surjective and hence the restriction
${\mathcal U}(\mathcal G_{\Gamma}^{\rho})\to \mathcal U(\mathcal A_{\Gamma})$ is also surjective.
By Lemma \ref{hut}, we have 
\[{\rm rank}\,\Gamma\ge \dim {\mathcal U}(\mathcal G_{\Gamma}^{\rho})\ge\dim\mathcal U(\mathcal A_{\Gamma})={\rm rank}\,\Gamma
\]
and so $\dim {\mathcal U}(\mathcal G_{\Gamma}^{\rho})=\dim\mathcal U(\mathcal A_{\Gamma})$.
Hence the map ${\mathcal U}(\mathcal G_{\Gamma}^{\rho})\to \mathcal U(\mathcal A_{\Gamma})$ induces an isomorphism of Lie algebras.
By using the exponential maps, we can show that the map ${\mathcal U}(\mathcal G_{\Gamma}^{\rho})\to \mathcal U(\mathcal A_{\Gamma})$ is also an isomorphism.
Hence the second and third assertions hold.

The fourth  follows from Proposition \ref{miniex}.

\end{proof}

For a simply connected solvable Lie group $G$,
we can directly construct the algebraic hull $\mathcal A_{G}$ of $G$ by using the Lie algebra $\g$.
Let $\n$ be the nilradical (i.e. maximal nilpotent ideal) of $\g$.
There exists a subvector space (not necessarily Lie algebra) $V$ of $\g$ so that
$\g=V\oplus \n$ as the direct sum of vector spaces and for any  $A,B\in V$ $({\rm ad}_A)_{s}(B)=0$ where $(ad_A)_{s}$  is the semi-simple part of ${\rm ad}_{A}$ (see \cite[Proposition I\hspace{-.1em}I\hspace{-.1em}I.1.1] {DER}).
We define the  map ${\rm ad}_{s}:\g\to D(\g)$ as 
${\rm ad}_{sA+X}=({\rm ad}_{A})_{s}$ for $A\in V$ and $X\in \n$.
Then we have $[{\rm ad}_{s}(\g), {\rm ad}_{s}(\g)]=0$ and ${\rm ad}_{s}$ is linear (see \cite[Proposition I\hspace{-.1em}I\hspace{-.1em}I.1.1] {DER}).
Since we have $[\g,\g]\subset \n$,  the  map ${\rm ad}_{s}:\g\to D(\g)$ is a representation and the image ${\rm ad}_{s}(\g)$ is Abelian and consists of semi-simple elements.
We take the Lie group homomorphism ${\rm Ad}_{s}:G\to {\rm Aut}(\g)$ which corresponds to the Lie algebra homomorphism  ${\rm ad}_{s}$.
We define the nilpotent Lie algebra 
\[\frak u=\{X-{\rm ad}_{sX}\in {\rm Im}\, {\rm ad}_{s}\ltimes \g\}
\]
which satisfies ${\rm Im}\, {\rm ad}_{s}\ltimes \g={\rm Im}\, {\rm ad}_{s}\ltimes \frak u$.
By the exponential map, we have the unipotent $\R$-defined algebraic group $\mathcal U$ corresponding to the Lie algebra $\frak u$.
We have ${\rm Ad}_{s}(G)\subset {\rm Aut}(\mathcal U)$ and $ {\rm Aut}(\mathcal U)$ is an $\R$-defined algebraic group.
Take the Zariski-closure $\mathcal T$ of ${\rm Ad}_{s}(G)$ in ${\rm Aut}(\mathcal U)$.
\begin{proposition}{\rm (\cite[Proposition 2.4]{Kas})}\label{PAL} 
The $\R$-defined algebraic group $\mathcal T\ltimes \mathcal U$ is an algebraic hull of $G$.
\end{proposition}

Suppose $G$ admits a lattice (i.e. cocompact discrete subgroup) $\Gamma$.
A discrete subgroup $\Gamma$ of a simply connected solvable Lie group $G$ is torsion-free polycyclic with ${\rm rank} \, \Gamma\le \dim G$.
It is known that $\Gamma$ is a lattice  if and only if ${\rm rank} \, \Gamma= \dim G$ (see \cite[Lemma 3.4]{OB}).
It is known  that  for  the algebraic hull   $\mathcal A_{G}$ of $G$,
 the Zariski-closure of $\Gamma$ in $\mathcal A_{G}$  is isomorphic to  the algebraic hull $\mathcal A_{\Gamma}$ of $\Gamma$ as an $\R$-defined algebraic group (see \cite[Proof of Theorem 4.34]{R}).
But it does not coincide with  $\mathcal A_{G}$ in general.

Let $k= \R$ or $\C$ and $\rho :G\to GL(V_{k})$  a finite-dimensional representation.
Consider the $\rho$-relative algebraic hull  $\mathcal G_{G}^{\rho}$.
Since the Zariski-closure of $\Gamma$ in $\mathcal A_{G}$  is   the algebraic hull $\mathcal A_{\Gamma}$ of $\Gamma$ as an $\R$-defined algebraic group,
the Zariski-closure of $\Gamma$ in $\mathcal G_{G}^{\rho}$ is identified with the $\rho$-relative algebraic hull  $\mathcal G_{\Gamma}^{\rho}$.
Hence we have the commutative diagram
\[\xymatrix{
\mathcal G_{\Gamma}^{\rho}\ar[r]&\mathcal G_{G}^{\rho}\\
\Gamma\ar[u]^{i_{G}\oplus \rho}\ar[r]& G\ar[u]_{i_{G}\oplus \rho}
}
\]
and each map is an inclusion.
We have $\mathcal U(\mathcal G_{G}^{\rho})=\mathcal U(\mathcal G_{\Gamma}^{\rho})$ but we do not have  $\mathcal G_{\Gamma}^{\rho}=\mathcal G_{G}^{\rho}$ in general.
\begin{proposition}\label{mok}
The following two conditions are equivalent.
\begin{itemize}
\item $\mathcal G_{\Gamma}^{\rho}=\mathcal G_{G}^{\rho}$.

\item $\rho$ is $\Gamma$-admissible i.e.  for the representation $
 {\rm Ad} \oplus \rho :G\to {\rm Aut}(\g\otimes\C) \times  GL(V)$,  
$({\rm Ad}\oplus\rho)(G)$ and $({\rm Ad}\oplus\rho)(\Gamma)$ have the same Zariski-closure in $ {\rm Aut}(\g\otimes \C) \times  GL(V)$.
\end{itemize}
\end{proposition}
\begin{proof}

Let $\mathcal F$ be the Zariski-closure of ${\rm Ad}(G)$ in $ {\rm Aut}(\g\otimes \C)$.
Then the Zariski-closure $\mathcal T$ of ${\rm Ad}_{s}(G)$ is a maximal torus of  $\mathcal F$ (see \cite[Proof of Proposition 3.2]{Kas}).
Hence, for a splitting $\mathcal F=\mathcal T\ltimes \mathcal U(\mathcal F)$ with the projection $p:\mathcal F\to \mathcal T$,
the map ${\rm Ad}_{s}:G\to \mathcal T$ is considered as
the composition
\[\xymatrix{
G\ar[r]^{\rm Ad}& \mathcal F\ar[r]^{p}&\mathcal T
}.
\]
On the other hand, by Proposition \ref{PAL}, $\mathcal T $
 is a maximal torus of the algebraic hull $\mathcal A_{G}=\mathcal T\ltimes \mathcal U$
and hence ${\rm Ad}_{s}$ is the 
composition
\[\xymatrix{
G\ar[r]& \mathcal A_{G}=\mathcal T\ltimes \mathcal U\ar[r]&\mathcal T
}.
\]
Let $\mathcal F^{\prime}$ be the  Zariski-closure of $\rho(G)$ in $GL(V)$.
Take a maximal torus $\mathcal T^{\prime}$ and a splitting 
$\mathcal F^{\prime}=\mathcal T^{\prime}\ltimes \mathcal U(\mathcal F^{\prime})$ with the projection $q:\mathcal F^{\prime}\to \mathcal T^{\prime}$.
Consider ${\rm Ad}_{s}\oplus q\circ \rho\to \mathcal T\times \mathcal T^{\prime}$.
Then the Zariski-closure of $({\rm Ad}_{s}\oplus q\circ \rho)(G)$ is a maximal torus of  both the Zariski-closures of  $({\rm Ad}\oplus  \rho)(G)$ and $(i_{G}\oplus \rho)(G)$.

Let $\sigma : G\to GL(V)$ a finite-dimensional representation.
Take the Zariski-closures $\mathcal G$ and $\mathcal G^{\prime}$ of $\sigma(G)$ and $\sigma(\Gamma)$.
Then we have $\mathcal U(\mathcal G)=\mathcal U(\mathcal G^{\prime})$ (see \cite[Theorem 3.2]{R}).
Hence, for a maximal torus $\mathcal S$ of $\mathcal G$, 
$\mathcal G=\mathcal G^{\prime}$ if and only if $\mathcal S\subset\mathcal G^{\prime}$.
Taking a splitting $\mathcal G=\mathcal S\ltimes \mathcal U(\mathcal G)$ with the projection $r: \mathcal G\to \mathcal S$,
$\mathcal G=\mathcal G^{\prime}$ if and only if
$r\circ \sigma(\Gamma)$ is Zariski-dense in $\mathcal S$.
By this argument,
the two conditions in the proposition are equivalent to condition that $({\rm Ad}_{s}\oplus  q\circ \rho)(G)$  and $({\rm Ad}_{s}\oplus q\circ   \rho)(\Gamma)$ have the same Zariski-closure.
Hence the proposition holds.
\end{proof}

\section{Cohomology of algebraic groups}
Let $k$ be a subfield of $\C$.
Let $\mathcal G$ be a $k$-defined algebraic group and $\mathcal H$ a normal subgroup of $\mathcal G$.
Consider the functor $V\to V^{G}$ from the category of $k$-rational $\mathcal G$-modules to the category of $k$-vector spaces.
Regarding this functor as the composition
$V\to V^{\mathcal H}\to (V^{\mathcal H})^{\mathcal G/\mathcal H}$, we can obtain the spectral sequence $E_{\ast}^{\ast,\ast}$ such that
$E_{2}^{\ast,\ast}=H^{\ast}_{r}(\mathcal G/\mathcal H, H^{\ast}_{r}(\mathcal H,V))$ and it converges to $H^{\ast}(\mathcal G,V)$ (see \cite[Proposition 6.6]{Jan}).
Considering the  spectral sequence for   the unipotent radical $\mathcal U(\mathcal G)$,
we have the following result see \cite{Haia}.
\begin{theorem}\label{unirediso}
For a $k$-rational $\mathcal G$-module $V$, we have an isomorphism
\[H^{\ast}_{r}({\mathcal G}, V_{k})\cong
 H^{\ast}_{r}({\mathcal U}({\mathcal G}), V_{k})^{\mathcal G/{\mathcal U}(k)}.
\]
\end{theorem}

The  extension
\[\xymatrix{
1\ar[r]&{\mathcal U}({\mathcal G})\ar[r]&\mathcal G\ar[r]&\mathcal G/{\mathcal U}({\mathcal G})\ar[r]&1
}
\]
splits  (see \cite{Mos1}).
Hence we have an affine action $\mathcal G\to {\rm Aut} \left({\mathcal U}({\mathcal G})\right)\ltimes {\mathcal U}({\mathcal G})$.
Consider the  coordinate ring $k[{\mathcal U}({\mathcal G})]$  as a $k$-rational $\mathcal G$-module.
Then by \cite[Part I. 4.7]{Jan}, for a $k$-rational $\mathcal G$-module $V$, we have
\[H^{\ast}({\mathcal G}, k[{\mathcal U}({\mathcal G})]\otimes V_{k})\cong
H^{\ast}({\mathcal U}({\mathcal G}),k[{\mathcal U}({\mathcal G})]\otimes V_{k})^{\mathcal G/{\mathcal U}({\mathcal G})(k)}\cong \left\{\begin{array}{ccccc}
  V_{k}^{\mathcal G/{\mathcal U}({\mathcal G})(k)} & \ast=0 \\
0& \ast>0 
\end{array}
\right. .
\]
Hence we have:
\begin{corollary}\label{VVVk}
For $\ast>0 $,  we have 
\[H^{\ast}_{r}({\mathcal G}, k[{\mathcal U}({\mathcal G})]\otimes V_{k})=0.
\]
\end{corollary}

Let $\frak u_{k}$ be the $k$-Lie algebra of ${\mathcal U}({\mathcal G})$.
For a $k$-rational $\mathcal G$-module $V$, 
we consider the Lie algebra cohomology $H^{\ast}(\frak u_{k},V_{k})$.
\begin{theorem}{\rm (\cite[Theorem 5.2]{Hoc})}\label{hocun}
We have a natural isomorphism
\[H^{\ast}_{r}(\mathcal G, V_{k})\cong H^{\ast}(\frak u_{k}, V_{k} )^{\mathcal G/\mathcal U(\mathcal G)(k)}.
\]

\end{theorem}

\section{de Rham cohomology of solvmanifolds  and Lie algebra cohomology}\label{MO}

Let $G$ be a simply connected solvable Lie group with a lattice $\Gamma$ and $V$ a finite-dimensional $G$-module for a Lie group representation $\rho:G\to  GL(V)$ of a finite-dimensional complex vector space $V$.
We consider the solvmanifold $\Gamma\backslash G$.
Since  we have $\pi_{1}(\Gamma\backslash G)\cong  \Gamma$,
 we have a flat vector bundle $E$ with flat connection $D$ over $\Gamma\backslash G$ whose monodromy is $\rho:\Gamma\to GL(V)$.
Denote by $A^{\ast}(\Gamma\backslash G, E)$ the cochain complex of $E$-valued differential forms on $\Gamma\backslash G$ with the differential $D$.
Since the solvmanifold $\Gamma\backslash G$ is an Eilenberg–MacLane space with the fundamental group $\Gamma$,
the de Rham cohomology $H^{\ast}(\Gamma\backslash G,E)$ is isomorphic to the group cohomology $H^{\ast}(\Gamma,V)$.
Let $\g$ be the Lie algebra of $G$ and 
 $\bigwedge \g^{\ast}\otimes V$ be the cochain complex of the Lie algebra $\g$ with values in the $\g$-module $V$.
Then  we regard $\bigwedge \g^{\ast}\otimes V$  as the cochain complex of left-$G$-invariant differential forms on $\Gamma\backslash G$  and we consider the inclusion
\[\iota:\bigwedge \g^{\ast}\otimes V\hookrightarrow A^{\ast}(\Gamma\backslash G, E).
\]
\begin{theorem}{\rm (\cite[Remark 7.30]{R})}
The inclusion
\[\iota:\bigwedge \g^{\ast}\otimes V\hookrightarrow A^{\ast}(\Gamma\backslash G, E)
\]
induces an  injection
\[H^{\ast}(\g,V)\hookrightarrow H^{\ast}(\Gamma\backslash G,E).
\]
\end{theorem}
The well-known proof of this theorem is given by geometric  techniques.
As we consider  $A^{\ast}(\Gamma\backslash G, E)=\mathcal C^{\infty}(\Gamma\backslash G)\otimes \bigwedge \g^{\ast}\otimes V$,
by using a normalized Haar measure on $G$ and integration on $\Gamma\backslash G$, we can construct the cochain complex homomorphism
\[\mu: A^{\ast}(\Gamma\backslash G, E)\to \bigwedge \g^{\ast}\otimes V
\]
such that $\mu \circ \iota ={\rm id}$.
In this paper, by using our theorems, we will give algebraic proof of this theorem.

As in Proposition \ref{mok}, we call a representation $\rho$ $\Gamma$-admissible if for the representation $
 {\rm Ad} \oplus \rho :G\to {\rm Aut}(\g\otimes\C) \times  GL(V)$,  
$({\rm Ad}\oplus\rho)(G)$ and $({\rm Ad}\oplus\rho)(\Gamma)$ have the same Zariski-closure in $ {\rm Aut}(\g\otimes \C) \times  GL(V)$.
\begin{theorem}\label{Moss}{\rm (\cite{Mosc},\cite[Theorem 7.26]{R})} 
If $\rho$ is $\Gamma$-admissible, then the inclusion 
\[\iota:\bigwedge \g^{\ast}\otimes V\hookrightarrow A^{\ast}(\Gamma\backslash G, E)
\]
induces a cohomology isomorphism
\[H^{\ast}(\g,V)\cong  H^{\ast}(\Gamma\backslash G,E).
\]
\end{theorem}
We will see that our theorems can be regarded as generalizations of this theorem.
We note that we have an isomorphism $H^{\ast}_{c}(G,V)\cong H^{\ast}(\g,V)$ and the map $H^{\ast}(\g,V)\to  H^{\ast}(\Gamma\backslash G,E)$
induced by the inclusion $\iota:\bigwedge \g^{\ast}\otimes V\hookrightarrow A^{\ast}(\Gamma\backslash G, E)$
is identified with the map $H^{\ast}_{c}(G,V)\to H^{\ast}(\Gamma,V)$ induced by the inclusion $\Gamma\subset G$.

\section{Baues' aspherical manifolds}\label{BAUSS}
Let $\Gamma$ be a torsion-free virtually polycyclic group and
$\mathcal A_{\Gamma}$  the $\Q$-algebraic hull of $\Gamma$.
Then $\mathcal A_{\Gamma}(\R)$ decomposes as a semi-direct product $\mathcal A_{\Gamma}(\R)=\mathcal A_{\Gamma}/\mathcal U(\mathcal A_{\Gamma})(\R)\ltimes \mathcal U(\mathcal A_{\Gamma})(\R)$.
Let $\frak{u}_{\R}$ be the Lie algebra of $\mathcal  U(\mathcal A_{\Gamma})(\R)$. 
Since the exponential map ${\exp}:{\frak u}_{\R} \longrightarrow \mathcal  U(\mathcal A_{\Gamma})(\R)$ is a diffeomorphism, $\mathcal  U(\mathcal A_{\Gamma})(\R)$ is diffeomorphic to $\R^n$ such that $n={\rm rank}\,\Gamma$.
For the semi-direct product $\mathcal A_{\Gamma}(\R)=\mathcal A_{\Gamma}/\mathcal U(\mathcal A_{\Gamma})(\R)\ltimes \mathcal U(\mathcal A_{\Gamma})(\R)$ and the inclusion $\Gamma\subset \mathcal A_{\Gamma}(\R)$,
we have the $\Gamma$-action on $\mathcal U(\mathcal A_{\Gamma})(\R)=\R^n$.

In \cite{B} Baues showed that  such $\Gamma$-action is properly discontinuous and 
the quotient $M_{\Gamma}=\Gamma\backslash \mathcal \R^{n}$ is a compact  aspherical manifold with $\pi_{1}(M_{\Gamma})=\Gamma$.
We call $M_{\Gamma}$ a standard $\Gamma$-manifold.
Moreover,  Baues showed that a standard $\Gamma$-manifold is unique up to diffeomorphism and 
an infra-solvmanifold with the fundamental group $\Gamma$ is diffeomorphic to the standard $\Gamma$-manifold $M_{\Gamma}$.

Let $A^{\ast}(M_{\Gamma})$ be the de Rham complex of $M_{\Gamma}$.
Then $A^{\ast}(M_{\Gamma}) $ is  the set of   the $\Gamma$-invariant differential forms ${A^{\ast}(\R^{n})}^{\Gamma}$ on $\R^{n}$. 
Let $(\bigwedge \frak{u}_{\R} ^{\ast})^{\mathcal A_{\Gamma}/\mathcal U(\mathcal A_{\Gamma})(\R)}$ be the left-invariant forms on $\R^n=\mathcal U(\mathcal A_{\Gamma})(\R)$ which are fixed by $\mathcal A_{\Gamma}/\mathcal U(\mathcal A_{\Gamma})(\R)$.
Since $\Gamma\subset \mathcal A_{\Gamma}(\R)$, we have the inclusion
\[(\bigwedge \frak{u}_{\R} ^{\ast})^{\mathcal A_{\Gamma}/\mathcal U(\mathcal A_{\Gamma})(\R)} ={A^{\ast}(\R^{n})}^{\mathcal A_{\Gamma}(\R)} \subset {A^{\ast}(\R^n)}^{\Gamma}= A^{\ast}(M_{\Gamma}).\]
By using the geometry of an infra-solvmanifold, Baues showed:

\begin{theorem}{\rm (\cite[Theorem 1.8]{B})}\label{a-t-4}
This inclusion $(\bigwedge \frak{u}_{\R} ^{\ast})^{\mathcal A_{\Gamma}/\mathcal U(\mathcal A_{\Gamma})(\R)}\subset A^{\ast}(M_{\Gamma})$
 induces a cohomology isomorphism
\[H^{\ast}(\frak u)^{\mathcal A_{\Gamma}/\mathcal U(\mathcal A_{\Gamma})(\R)}\cong H^{\ast}(M_{\Gamma},\R).
\]
\end{theorem}
By Theorem \ref{hocun}, since $M_{\Gamma}$ is  an  aspherical manifold with $\pi_{1}(M_{\Gamma})=\Gamma$, we can say the following fact.
\begin{corollary}
We have an isomorphism
\[H^{\ast}_{r}(\mathcal A_{\Gamma},\R)\cong H^{\ast}(\Gamma,\R).
\]
\end{corollary}
Our main theorems give this fact without using the geometry of an infra-solvmanifold and generalize this isomorphism for the cohomology with values in modules.

\section{Corollaries of Theorem \ref{poly1} and \ref{lie1}}\label{cor}
\subsection{Cohomologies of torsion-free virtually polycyclic groups and simply connected solvable Lie groups}
Let $k$ be a subfield of $\C$  (resp. $k= \R$ or $\C$).
Let $\Gamma$ (resp. $G$) be a torsion-free virtually polycyclic group (resp. simply connected solvable Lie group) and
$\rho :\Gamma\to GL(V_{k})$ (resp. $\rho :G\to GL(V_{k})$)  a finite-dimensional representation.
We consider the algebraic hull of $\mathcal A_{\Gamma}$ (resp. $\mathcal A_{G}$) and the $\rho$-relative algebraic hull $\mathcal G_{\Gamma}^{\rho}$ (resp. $\mathcal G_{G}^{\rho}$) .
Let $\frak u_{k}$ be the $k$-Lie algebra of  ${\mathcal U}(\mathcal A_{\Gamma})$ (resp. ${\mathcal U}(\mathcal A_{G})$).
 It is also the  Lie algebra of ${\mathcal U}(\mathcal G_{\Gamma}^{\rho})$ (resp. ${\mathcal U}(\mathcal G_{G}^{\rho})$)  by Proposition \ref{relproa}.

By Theorem \ref{poly1}, Theorem  \ref{lie1},  Proposition \ref{relproa} and Proposition \ref{hocun} we have the following results.
\begin{corollary}\label{pppis}
For a $k$-rational $\mathcal G_{\Gamma}^{\rho}$-module (resp. $\mathcal G_{G}^{\rho}$-module)  $W$ (e.g. $W=V$), we have  isomorphisms
\[H^{\ast}(\Gamma,W_{k})\cong H^{\ast}_{r}(\mathcal G_{\Gamma}^{\rho},W_{k})\cong H^{\ast}(\frak u_{k},W_{k})^{\mathcal G_{\Gamma}^{\rho}/{\mathcal U}(\mathcal G_{\Gamma}^{\rho})(k)}.
\]
(resp. 
\[H^{\ast}_{c}(G,W_{k})\cong H^{\ast}_{r}(\mathcal G_{G}^{\rho},W_{k})\cong H^{\ast}(\frak u_{k},W_{k})^{\mathcal G_{G}^{\rho}/{\mathcal U}(\mathcal G_{G}^{\rho})(k)}.
\])
\end{corollary}

By the fourth assertion of Proposition \ref{relproa},
under the good condition, we can simplify the statement and imply  Baues' isomorphism.

\begin{corollary}
If the Zariski-closure of $\rho(\Gamma)$ (resp. $\rho(G)$) is minimal,  we have  isomorphisms
\[H^{\ast}(\Gamma,V_{k})\cong H^{\ast}_{r}(\mathcal A_{\Gamma},V_k)\cong H^{\ast}(\frak u_{k},V_k)^{\mathcal A_{\Gamma}/{\mathcal U}(\mathcal A_{\Gamma})(k)}.
\]
(resp.
\[H^{\ast}_{c}(G,V_{k})\cong H^{\ast}(\mathcal A_{G},V_k)\cong H^{\ast}(\frak u_{k},V_k)^{\mathcal A_{G}/{\mathcal U}(\mathcal A_{G})(k)}.
\])
\end{corollary}

\begin{remark}
This statement gives an isomorphism
\[H^{\ast}(\Gamma,\Q)\cong H^{\ast}_{r}(\mathcal A_{\Gamma},\Q)\cong H^{\ast}(\frak u_{\Q},\Q)^{\mathcal A_{\Gamma}/{\mathcal U}(\mathcal A_{\Gamma})(\Q)}.
\]
For the standard $\Gamma$-manifold $M_{\Gamma}$,
in \cite[Proposition 13.6]{BG}, Baues  and   Grunewald claim the isomorphism
\[H^{\ast}(M_\Gamma,\Q)\cong H^{\ast}(\frak u_{\Q},\Q)^{\mathcal A_{\Gamma}/{\mathcal U}(\mathcal A_{\Gamma})(\Q)}.
\]
But they do not give an explicit proof. 
\end{remark}

\begin{remark}
The isomorphism in Theorem \ref{pppis} is absolutely algebraic.
In Theorem \ref{BEXT}, we will prove an extension of Baues' isomorphism from a geometric viewpoint.
\end{remark}




\begin{remark}
By Proposition \ref{PAL}, the Lie algebra of ${\mathcal U}(\mathcal A_{G})$ is given by
\[\frak u=\{X-{\rm ad}_{sX}\in {\rm Im}\, {\rm ad}_{s}\ltimes \g\}
\]
as in Section \ref{sefull}.
Since ${\mathcal U}(\mathcal A_{G})$ is  unipotent,
a rational ${\mathcal U}(\mathcal A_{G})$-module is unipotent and hence it is a nilpotent $\frak u$-module.
Hence this corollary gives  "nilpotentizations" of cohomology of solvable Lie algebras.

\end{remark}

\begin{corollary}
Let $\g$ be a  solvable real Lie algebra.
Take the nilpotent Lie algebra
\[\frak u=\{X-{\rm ad}_{sX}\in {\rm Im}\, {\rm ad}_{s}\ltimes \g\}
\]
as in Section \ref{sefull}.
Then for a finite-dimensional real $\g$-module $V$,
taking  a certain nilpotent $\frak u$-module structure on $V$,
 the Lie algebra cohomology $H^{\ast}(\g,V)$ is isomorphic to a subspace of Lie algebra cohomology $H^{\ast}(\frak u,V)$. 
\end{corollary}

\subsection{Cohomology of lattices in simply connected solvable Lie groups}
In this subsection,  $k= \R$ or $\C$.
Let $G$ be a simply connected solvable Lie group with a lattice $\Gamma$.
Let $\rho :G\to GL(V_{k})$ be a finite-dimensional representation.
We consider the commutative diagram
\[\xymatrix{
\mathcal G_{\Gamma}^{\rho}\ar[r]&\mathcal G_{G}^{\rho}\\
\Gamma\ar[u]^{i_{G}\oplus \rho}\ar[r]& G\ar[u]_{i_{G}\oplus \rho}
}
\]
as in Section \ref{sefull}.
Then by our results, the inclusions $i_{G}\oplus \rho:\Gamma\to \mathcal G_{\Gamma}^{\rho}$ and $i_{G}\oplus \rho:G\to \mathcal G_{G}^{\rho}$ induce isomorphisms
\[H^{\ast}(\Gamma,V_{k})\cong H^{\ast}_{r}(\mathcal G_{\Gamma}^{\rho},V_{k})\cong H^{\ast}_{r}(\mathcal U(\mathcal G_{\Gamma}^{\rho}),V_{k})^{\mathcal G_{\Gamma}^{\rho}/\mathcal U(\mathcal G_{\Gamma}^{\rho})(k)}\]
and
\[H^{\ast}_{c}(G,V_{k})\cong H^{\ast}_{r}(\mathcal G_{G}^{\rho},V_{k})\cong H^{\ast}_{r}(\mathcal U(\mathcal G_{G}^{\rho}),V_{k})^{\mathcal G_{G}^{\rho}/\mathcal U(\mathcal G_{G}^{\rho})(k)}.\]
Since we have $\mathcal U(\mathcal G_{\Gamma}^{\rho})=\mathcal U(\mathcal G_{G}^{\rho})$,
the induced map $H^{\ast}_{c}(G,V_{k})\to H^{\ast}(\Gamma,V_{k})$ is identified with the inclusion
\[H^{\ast}_{r}(\mathcal U(\mathcal G_{G}^{\rho}),V_{k})^{\mathcal G_{G}^{\rho}/\mathcal U(\mathcal G_{G}^{\rho})(k)}\subset H^{\ast}_{r}(\mathcal U(\mathcal G_{\Gamma}^{\rho}),V_{k})^{\mathcal G_{\Gamma}^{\rho}/\mathcal U(\mathcal G_{\Gamma}^{\rho})(k)}.
\]
Hence our theorem gives the following result without using the geometry of the solvmanifold $\Gamma\backslash G$.
\begin{corollary}
The inclusion $\Gamma\to G$ induces an injection
\[H^{\ast}_{c}(G,V_{k})\hookrightarrow H^{\ast}(\Gamma,V_{k}).
\]
\end{corollary}
Moreover, as the following result,
by Proposition \ref{mok},  our theorems imply the Mostow theorem.
\begin{corollary}
If  $\mathcal G_{\Gamma}^{\rho}=\mathcal G_{G}^{\rho}$ (equivalently $\rho$ is $\Gamma$-admissible),
then we have 
\[H^{\ast}_{r}(\mathcal U(\mathcal G_{G}^{\rho}),V_{k})^{\mathcal G_{G}^{\rho}/\mathcal U(\mathcal G_{G}^{\rho})(k)}= H^{\ast}_{r}(\mathcal U(\mathcal G_{\Gamma}^{\rho}),V_{k})^{\mathcal G_{\Gamma}^{\rho}/\mathcal U(\mathcal G_{\Gamma}^{\rho})(k)}.
\]
and hence $\Gamma\to G$ induces an isomorphism
\[H^{\ast}_{c}(G,V_{k})\cong H^{\ast}(\Gamma,V_{k}).
\]

\end{corollary}

\subsection{Nilpotent Case}
Let $\Gamma$ be a torsion-free nilpotent group.
Then we have (see \cite[Chapter II]{R} and \cite[Chapter 2]{OV}):
\begin{itemize}
\item $\Gamma$ is a lattice of some simply connected nilpotent Lie group $N$.
\item $N$ can be the group $\mathcal U(\R)$ of real points of a $\Q$-defined algebraic unipotent group $\mathcal U$.
\item $\Gamma\subset \mathcal U (\Q)$ and $\Gamma$ is Zariski-dense in $\mathcal U$.
\end{itemize}
 We denote by $\frak u_{k}$ the $k$-Lie algebra of $\mathcal U$.
Then $\frak u_{\R}$ is the Lie algebra of $N$.
By our results, we have:
\begin{corollary}
For a $k$-rational $\mathcal U$-module $V$, the inclusion  $\Gamma\subset \mathcal U (\Q)$ induces  cohomology isomorphisms
\[H^{\ast}(\frak u_{k},V_{k})\cong H^{\ast}(\mathcal U, V_{k})\cong H^{\ast}(\Gamma, V_{k}).
\]
\end{corollary}
\begin{remark}
This corollary is an algebraic analogy with the results given by Nomizu \cite{Nom} and  Lambe-Priddy \cite{LP}.
An isomorphism
\[H^{\ast}(\frak u_{\R}, \R)\cong  H^{\ast}(\Gamma, \R)
\]
was given by the de Rham cohomology of the nilmanifold $\Gamma\backslash N $ in \cite{Nom}.
Moreover  an
 isomorphism
\[H^{\ast}(\frak u_{\Q}, \Q)\cong  H^{\ast}(\Gamma, \Q)
\]
was given by the simplicial rational de Rham cohomology of the simplicial classifying complex  of $\Gamma$ in \cite{LP}.
\end{remark}

\section{New proof of the Dekimpe-Igodt surprising cohomology vanishing theorem}\label{DDEKimp}

Let $\Gamma$ (resp. $G$) 
be a torsion-free virtually polycyclic group (resp. simply connected solvable Lie group) and
$\rho :\Gamma\to GL(V_{k})$ (resp. $\rho :G\to GL(V_{k})$)
 a finite-dimensional representation.
Consider the $\rho$-relative algebraic hull $\mathcal G_{\Gamma}^{\rho}$ (resp $\mathcal G_{G}^{\rho}$).
Take a splitting $\mathcal G_{\Gamma}^{\rho}=\mathcal G_{\Gamma}^{\rho}/\mathcal U(\mathcal G_{\Gamma}^{\rho})\ltimes \mathcal U(\mathcal G_{\Gamma}^{\rho})$  (resp. $\mathcal G_{G}^{\rho}=\mathcal G_{G}^{\rho}/\mathcal U(\mathcal G_{G}^{\rho})\ltimes \mathcal U(\mathcal G_{G}^{\rho})$).
We simply write $\mathcal T=\mathcal G_{\Gamma}^{\rho}/\mathcal U(\mathcal G_{\Gamma}^{\rho})$ (resp. $\mathcal G_{G}^{\rho}/\mathcal U(\mathcal G_{G}^{\rho})$)
 and $\mathcal U=\mathcal U(\mathcal G_{\Gamma}^{\rho})$. (resp. $\mathcal U(\mathcal G_{G}^{\rho})$).
Then, considering $W_{k}=k[\mathcal U]\otimes V_{k}$ in Corollary \ref{pppis}, by Corollary \ref{VVVk},
we have the following result.
\begin{corollary}\label{VANNN}
For $\ast>0 $,  we have 
\[H^{\ast}(\Gamma, k[{\mathcal U}]\otimes V_{k})=0
\]
(resp. 
\[H^{\ast}(G, k[{\mathcal U}]\otimes V_{k})=0.)
\]
\end{corollary}

Since the exponential map $\exp: \frak u\to \mathcal U$ is an isomorphism of $k$-defined algebraic variety,
as we regard $\frak u_{k}=k^{r}$ with $r={\rm rank}\, \Gamma$ (resp. $\dim G$),
the coordinate ring $k[{\mathcal U}]$ can be regarded as the space $P(k^{r})$ of the $k$-polynomial functions on $\frak u_{k}=k^{r}$.

We assume $k=\R$.
A group action on $\R^{r}$ is called bounded polynomial  if for some integer $d$ the action is represented by $d$-bounded degree polynomial diffeomorphisms of  $\R^{r}$  with  $d$-bounded degree polynomial inverse.
For the $\R$-defined  algebraic group $\mathcal  G_{\Gamma}^{\rho}=\mathcal T\ltimes \mathcal U$,
considering $\frak u_{\R}=\R^{r}$,
we have the algebraic group action of $\mathcal G(\R)$ on $\R^{r}$.
By the fundamental theory of algebraic groups, this action is bounded polynomial.
Hence the $\Gamma$-action  on $\frak u_{\R}=\R^{r}$ is also bounded polynomial.
Take $\mathcal S=\mathcal Z(\mathcal U)\cap \mathcal T$ where $\mathcal Z(\mathcal U)$ is the centralizer of $\mathcal U$.
Then $\mathcal G(\R)/\mathcal S(\R)=(\mathcal T(\R)/\mathcal S(\R))\ltimes \mathcal U(\R)$ is the $\R$-points of  algebraic hull of $\Gamma$ 
with the Zariski-dense inclusion $\Gamma\to (\mathcal T(\R)/\mathcal S(\R))\ltimes \mathcal U(\R)$
(see \cite[Appendix A]{B}).
Hence the $\Gamma$-action on  $ \mathcal U(\R)$ is the action for the construction of Baues' standard $\Gamma$-manifold as Section \ref{BAUSS}.
This implies that  the $\Gamma$-action on $\frak u_{\R}=\R^{r}$ is bounded polynomial and crystallographic (i.e. properly discontinuous and cocompact).

In \cite{BD}, Benoist and Dekimpe showed that a bounded polynomial crystallographic  action  of  $\Gamma$ is unique up to conjugation by a bounded polynomial diffeomorphism.

Hence Corollary \ref{VANNN} says the following statement.
\begin{theorem}\label{vanppp}
Let $\Gamma$ be  a torsion-free virtually polycyclic group.
Suppose  $\Gamma$  admits a bounded polynomial crystallographic  action   on $\R^{r}$.
Consider the vector space $P(\R^{r})$ of the polynomial functions on $\R^{r}$ as a $\Gamma$-module.
Then for any   representation $\rho:\Gamma\to GL(V_{\R})$  of a finite-dimensional $\R$-vector space $V$, for $\ast> 0$ we have
\[H^{\ast}(\Gamma, P(\R^{r})\otimes V_{\R})=0.\]
\end{theorem} 

\begin{remark}
\begin{enumerate}
\item
If a virtually polycyclic group $\Gamma$ is not torsion-free, then we can not construct the standard $\Gamma$-manifold.
But, for every virtually polycyclic group $\Gamma$, there exists a finite index normal torsion-free subgroup $\Gamma^{\prime}$ in $\Gamma$.
Hence, assuming that $\Gamma$  admits a bounded polynomial crystallographic  action   on $\R^{r}$,
 for $\ast>0$, by the vanishing $H^{\ast}(\Gamma^{\prime},P(\R^{r})\otimes V_{\R})=0$ for torsion-free case,
we can easily see
\[H^{\ast}(\Gamma,P(\R^{r})\otimes V_{\R})=0.
\]
(see \cite[Chapter III 10]{Bro})
\item
Denote by $P(\R^{r}, V_{\R})$ the vector  space of polynomial maps from $\R^{r}$ to $V_{\R}$.
It is known that  the map $P(\R^{r})\otimes V_{\R}\to P(\R^{r}, V_{\R})$
such that $f\otimes v\mapsto (\R^{r}\ni a\mapsto f(a)v)$ is an isomorphism (cf. \cite[page 16]{Hocin}).
Denote by ${\mathcal P}(\R^{r})$ the group of polynomial diffeomorphisms of $\R^{r}$.
Since  the map  $P(\R^{r})\otimes V_{\R}\to P(\R^{r}, V_{\R})$ is ${\mathcal P}(\R^{r}) \times GL(V_{\R})$-equivariant,
for a polynomial $\Gamma$-action on $\R^{r}$, we can identify the $\Gamma$-module $P(\R^{r})\otimes V_{\R}$ with the $\Gamma$-module $P(\R^{r}, V_{\R})$.

\item
In \cite{Deki},  for the special bounded polynomial action ( called "canonical type" ) of a  virtually polycyclic group $\Gamma$,  Dekimpe and Igodt showed the cohomology vanishing 
\[H^{\ast}(\Gamma,P(\R^{r}, V_{\R}))=0\] for $\ast>0$.
They said such cohomology vanishing is surprising.
The proof which is given by Dekimpe and Igodt is very hard (see \cite[Section 3]{Deki}).
Now we obtain a new proof of such vanishing theorem.
\end{enumerate}
\end{remark}

In this paper we give a new application of the vanishing theorem.
Take the standard $\Gamma$-manifold  $\Gamma\backslash \mathcal U(\R)$.
We consider  the space
\[\left(\bigwedge {\frak u_{\R}}^{\ast}\otimes V_{\R}\right)^{\mathcal T(\R)}\]  of $\mathcal T(\R)$-invariant elements of the cochain complex of the  $\R$-Lie algebra ${\frak u}_{\R}$ of $\mathcal U$  with values in the $\R$-rational $\mathcal U$-module $V$.
Then,  as Section \ref{BAUSS},  we have the natural inclusion
\[\left(\bigwedge {\frak u_{\R}}^{\ast}\otimes V_{\R}\right)^{\mathcal T(\R)}\subset A^{\ast}(M_{\Gamma}, E)
\]
where $A^{\ast}(M_{\Gamma}, E)$ is the de Rham complex with values in the flat bundle $E$ corresponding to $\rho :\Gamma\to GL(V_{\R})$.
We prove an extension of  Theorem \ref{a-t-4} for any finite-dimensional representation $\rho$ without using the geometry of an infra-solvmanifold. 
(On the other hand, we use polynomial geometry)
\begin{theorem}\label{BEXT}
The inclusion 
\[\left(\bigwedge {\frak u_{\R}}^{\ast}\otimes V_{\R}\right)^{\mathcal T(\R)}\subset A^{\ast}(M_{\Gamma}, E)
\] 
induces a cohomology isomorphism
\[H^{\ast}\left( {\frak u_{\R}}, V_{\R}\right)^{\mathcal T(\R)}\cong H^{\ast}(M_{\Gamma},E)
.\]
\end{theorem}
\begin{proof}
Let  $A^{\ast}(\mathcal U(\R), V)$ be the cochain complex of $V$-valued $\mathcal C^{\infty}$-differential forms on $\mathcal U(\R)$ and $A_{poly}^{\ast}(\mathcal U(\R), V)\subset A^{\ast}(\mathcal U(\R), V)$ be the subcomplex of   $V$-valued $\R$-polynomial differential forms.
Then as a $\R$-rational $\mathcal  G_{\Gamma}^{\rho}$-module,
we have
\[A_{poly}^{\ast}(\mathcal U(\R), V_{\R})=\bigwedge {\frak u_{\R}}^{\ast}\otimes V_{\R}\otimes \R[{\mathcal U}].
\]
Since $\Gamma$ is Zariski-dense in  $\mathcal  G_{\Gamma}^{\rho}$,
we have
\[A_{poly}^{\ast}(\mathcal U(\R), V)^{\Gamma}=A_{poly}^{\ast}(\mathcal U(\R), V_{\R})^{\mathcal  G_{\Gamma}^{\rho}}=\left(\bigwedge {\frak u_{\R}}^{\ast}\otimes V_{\R}\right)^{\mathcal T(\R)}.\]
By Corollary \ref{VANNN}, for any $s$, 
we have
\[H^{\ast}(\Gamma, A_{poly}^{s}(\mathcal U(\R), V_{\R}))=0
\]
for $\ast>0$.
Since $\mathcal U(\R)=\R^{r}$ by the exponential map,
we have \[H^{\ast}(A_{poly}^{\ast}(\mathcal U(\R), V_{\R}))=0\] for $\ast>0$.

Let  $\mathcal A^{\ast}(M_{\Gamma},E_{\Gamma})$ be the complex of sheaves of differential forms with values in $E$.
Consider the subcomplex $\mathcal A_{poly}^{\ast}(M_{\Gamma},E)$ of  $\mathcal A^{\ast}(M_{\Gamma},E)$  such that for an open set $U\in M_{\Gamma}$, the section
$\mathcal A_{poly}^{\ast}(M_{\Gamma},E)(U)$ of $U$ consists of the forms 
$\omega \in \mathcal A^{\ast}(M_{\Gamma},E)(U)$ 
so that $\pi^{\ast}\omega$ is a polynomial forms on $\pi^{-1}(U)$
where $\pi:\mathcal U(\R)\to M_{\Gamma}$ is the quotient map.
Then,
 for each $s$, for  small $U$ we have $\mathcal A_{poly}^{s}(M_{\Gamma},E)(U)=A_{poly}^{s}(\mathcal U(\R), V_{\R})$ and 
 the sheaf $\mathcal A_{poly}^{s}(M_{\Gamma},E)$ is the locally constant sheaf corresponding to  the $\Gamma$-module $A_{poly}^{\ast}(\mathcal U(\R), V_{\R})$.
Hence by   the above argument, the complex
$\mathcal A_{poly}^{\ast}(M_{\Gamma},E)$
is an acyclic resolution of the locally constant sheaf $\mathcal E$ defined by $E$.
We have the diagram
\[\xymatrix{
\mathcal E\ar[r]\ar@{=}[d]  &\mathcal A_{poly}^{0}(M_{\Gamma},E)\ar[r]\ar[d]  &\mathcal A_{poly}^{1}(M_{\Gamma},E)\ar[r] \ar[d]  & \dots\\
\mathcal E\ar[r]  &\mathcal A^{0}(M_{\Gamma},E)\ar[r] &\mathcal A^{1}(M_{\Gamma},E)\ar[r]  & \dots
}.
\]
Let $A_{poly}^{\ast}(M_{\Gamma},E)$ be the global section of $\mathcal A_{poly}^{s}(M_{\Gamma},E)(M)$.
By the standard argument of sheaf cohomology and the de Rham theorem (see e.g. \cite[Section 4]{Voi}),
the inclusion
\[A_{poly}^{\ast}(M_{\Gamma},E)\subset A^{\ast}(M_{\Gamma}, E) 
\]
induces a cohomology isomorphism.
By
\[\left(\bigwedge {\frak u_{\R}}^{\ast}\otimes V_{\R}\right)^{\mathcal T(\R)}=A_{poly}^{\ast}(\mathcal U(\R), V_{\R})^{\Gamma}= A_{poly}^{\ast}(M_{\Gamma},E),
\]
 the theorem follows.
\end{proof}

By Corollary \ref{VANNN},  we can also give the continuous cohomology version of the cohomology vanishing on a bounded  polynomial simply transitive action of a simply connected solvable Lie group $G$.
For the $\R$-defined  algebraic group $\mathcal  G_{G}^{\rho}=\mathcal T\ltimes \mathcal U$ with a Zariski-dense inclusion $G\subset \mathcal T\ltimes \mathcal U$, 
 the restriction of the projection $p:\mathcal T\ltimes \mathcal U\to \mathcal U$ on $G$ is a diffeomorphism onto $\mathcal U(\R)$ and  by this diffeomorphism,  the restricted action of $G$ on $
\mathcal U(\R)$ can be identified with the action of left translation on $G$ (see \cite[Section 2]{B}).
Hence as the above argument, the action of $G$ on $\frak u_{\R}=\R^{r}$ is bounded polynomial and simply transitive.
In \cite{BD}, Benoist and Dekimpe showed that a bounded polynomial  simply transitive action  of  $G$ is unique up to conjugation by a bounded polynomial diffeomorphism.
Hence we have the following result.
\begin{theorem}
Let $G$ be a simply connected solvable Lie group.
Suppose  $G$ admits a bounded polynomial  simply transitive action    on $\R^{r}$.
Consider the vector space $P(\R^{r})$ of the polynomial functions on $\R^{r}$ as a continuous $G$-module.
Then, for any   representation $\rho:G\to GL(V_{\R})$  of a finite-dimensional $\R$-vector space $V$, for $\ast> 0$, we have
\[H^{\ast}_{c}(G, P(\R^{r})\otimes V_{\R})=0.\]

\end{theorem}

\end{document}